\documentclass [11pt]{amsart}  
\usepackage{a4}                 
\usepackage{paralist}           
\usepackage{graphicx}
\usepackage{amssymb}                               
\usepackage{float}
\usepackage{amsmath}
\usepackage{psfrag}
\usepackage{MnSymbol}
\usepackage[all]{xy}
\usepackage{amsmath,amsthm}

\theoremstyle{plain}
\newtheorem{thm}{Theorem}[section]

\newtheorem{lem}[thm]{Lemma}
\newtheorem{prop}[thm]{Proposition}
\newtheorem{cor}[thm]{Corollary}

\theoremstyle{definition}
\newtheorem{defn}[thm]{Definition}
\newtheorem*{conj}{Conjecture}
\newtheorem{fur}{Problem}

\theoremstyle{remark}
\newtheorem*{rem}{Remark}
\newtheorem*{note}{Note}

\newtheoremstyle{TheoremNum}
        {\topsep}{\topsep}              
        {\itshape}                      
        {}                              
        {\bfseries}                     
        {.}                             
        { }                             
        {\thmname{#1}\thmnote{ \bfseries #3}}
    \theoremstyle{TheoremNum}
    \newtheorem{thmn}{Theorem}

\DeclareMathOperator{\id}{id}

\DeclareMathOperator{\tor}{Tor}

\DeclareMathOperator{\torord}{Tord}

\DeclareMathOperator{\ths}{th}

\DeclareMathOperator{\T}{T}
\DeclareMathOperator{\tf}{tf}

\DeclareMathOperator{\ord}{o}
\DeclareMathOperator{\prim}{prime}

\title{On torsion in finitely presented groups}
\author{Maurice Chiodo}
\date{\today}

\begin{document}

\let\thefootnote\relax\footnotetext{2010 \textit{AMS Classification:} 20F10, 03D40, 03D80.}
\let\thefootnote\relax
\let\thefootnote\relax\footnotetext{\textit{The author was supported by:} a University of Melbourne Overseas Research Experience Scholarship, the Italian FIRB ``Futuro in Ricerca'' project RBFR10DGUA\_002, and the Swiss National Science Foundation grant FN PP00P2-144681/1.}

\begin{abstract}
We give a uniform construction that, on input of a recursive presentation $P$ of a group, outputs a recursive presentation of a torsion-free group, isomorphic to $P$ whenever $P$ is itself torsion-free. We use this to re-obtain a known result, the existence of a universal finitely presented torsion-free group; one into which all finitely presented torsion-free groups embed. We apply our techniques to show that recognising embeddability of finitely presented groups is $\Pi^{0}_{2}$-hard, $\Sigma^{0}_{2}$-hard, and lies in $\Sigma^{0}_{3}$. We also show that the sets of orders of torsion elements of finitely presented groups are precisely the $\Sigma^{0}_{2}$ sets which are closed under taking factors.
\end{abstract}

\maketitle

\section{Introduction}

By a finite presentation $\langle X|R \rangle$ of a group we mean, as usual, a finite collection of generators $X$, together with a finite set $R$ of defining relations. A recursive (resp.~countably generated recursive) presentation $\langle X|R \rangle$ of a group is then a finite (resp.~countable) collection of generators $X$, together with a recursive enumeration of a possibly infinite set $R$ of defining relations. We use $\overline{P}$ to denote the group presented by a presentation $P$.

The Higman embedding theorem \cite{Hig emb} shows that every recursively presented group embeds into a finitely presented group. Moreover, this embedding can be made \emph{uniform}; there is an algorithm that takes any recursive presentation $P$ and outputs a finite presentation $Q$ and an explicit embedding $\phi: \overline{P} \hookrightarrow \overline{Q}$. 
This embedding theorem was used by Higman to show the existence of a \emph{universal} finitely presented group; one into which all finitely presented groups embed. By analysing Higman's embedding theorem, we prove: 
\\

\begin{thmn}[\ref{uni tor free}]
There is a universal finitely presented torsion-free group $G$. That is, $G$ is torsion-free, and for any finitely presented group $H$ we have that $H \hookrightarrow G$ if (and only if) $H$ is torsion-free.
\end{thmn}
\mbox{}

Theorem \ref{uni tor free} first appeared (as far as we are aware) in the appendix by Oleg V. Belegradek of the paper \cite{Bele}, Theorem A.1. He gives a different proof to ours, making use of arguments from model theory. Moreover, in \cite[Remark A.2]{Bele} he points out that theorem \ref{uni tor free} can also be proved along the lines we follow in the present paper.

Key to many of the important results in this work is the technical observation that the Higman embedding theorem can preserve the set of orders of torsion elements; we state this as theorem \ref{tor emb}. 
Every group $G$ has a unique torsion-free quotient through which all other torsion-free quotients factor (see corollary \ref{tf quot}); we call this the \emph{torsion-free universal quotient} $G^{\tf}$. By standard techniques in combinatorial group theory, we show in proposition \ref{tkill} the existence of an algorithm that takes any finite presentation $P$ and outputs a recursive presentation $P^{\tf}$ of the torsion-free universal quotient of $\overline{P}$. Theorem \ref{uni tor free} then follows by combining theorem \ref{tor emb} and proposition \ref{tkill}, similar to Higman's original construction of a universal finitely presented group.

In \cite{Lempp} it was shown by Lempp that the problem of recognising torsion-freeness for finitely presented groups is $\Pi^{0}_{2}$-complete in Kleene's arithmetic hierarchy (see \cite{Rogers}, or the introduction to \cite{Lempp}, for a description of $\Sigma_{n}^{0}$ sets, $\Pi_{n}^{0}$ sets, and Kleene's arithmetic hierarchy).  
Therefore the set of finitely presented subgroups of any universal torsion-free finitely presented group is 
$\Pi_{2}^{0}$-complete, and, in particular, not recursively enumerable. In \cite{Chiodo} we gave another proof of the existence of a finitely presented group whose set of finitely presented subgroups is not recursively enumerable, without the use of the results of Lempp \cite{Lempp} or Oleg Belegradek \cite{Bele}. Building on theorem \ref{uni tor free}, we show the following.
\\ 

\begin{thmn}[\ref{main1}]
For any recursive enumeration $P_{1}, P_{2}, \ldots$ of all finite presentations of groups, the set $K=\{ (i , j) \in \mathbb{N}^{2}\ |\ \overline{P}_{i} \hookrightarrow \overline{P}_{j}\}$ is $\Sigma^{0}_{2}$-hard, $\Pi^{0}_{2}$-hard, and has a $\Sigma^{0}_{3}$ description.
\end{thmn}
\mbox{}

We write $\torord(G)$ to denote the orders of non-trivial torsion elements of a group $G$, and say a set $A \subseteq \mathbb{N}$ is \emph{factor-complete} if it is closed under taking multiplicative factors (excluding $1$). Applying theorem \ref{tor emb} to an idea by Dorais in the comments to \cite{MO Wilton}, we give the following complete characterisation of sets which can occur as $\torord(G)$ for $G$ a finitely (or recursively) presented group:
\\

\begin{thmn}[\ref{torord equivalence}]
For a set of natural numbers $A$ the following are equivalent:
\\$(1)$ $A=\torord(G)$ for some finitely presented group $G$;
\\$(2)$ $A=\torord(G)$ for some countably generated recursively presented group $G$;
\\$(3)$ $A$ is a factor-complete $\Sigma_{2}^{0}$ set.
\end{thmn}
\mbox{}

It follows (corollary \ref{any}) that we can realise any $\Sigma_{2}^{0}$ set, up to one-one equivalence, as $\torord(G)$ for some finitely presented group $G$.
\\

\noindent \textbf{Acknowledgements:} The author wishes to thank Jack Button, Andrew Glass, Steffen Lempp, Vincenzo Marra, Chuck Miller and Rishi Vyas for their many useful conversations and comments which led to the overall improvement of this work. Thanks also go to Fran\c{c}ois Dorais, Stefan Kohl, Benjamin Steinberg and Henry Wilton for their thoughtful discussion on MathOverflow, which led to the addition of section \ref{comp torord} to this work. We thank Igor Belegradek for bringing the work \cite{Bele} to our attention, and to the anonymous referee for suggesting improvements. Finally, thanks must go to the late Greg Hjorth, whose suggestion of making contact with Steffen led to the eventual writing of this work.

\section{Preliminaries}\label{dunno}

\subsection{Notation}\mbox{}

\noindent With the convention that $\mathbb{N}$ contains $0$, we write $\varphi_{m}$ to be the $m^{\ths}$ \emph{partial recursive function} $\varphi_{m}: \mathbb{N} \to \mathbb{N}$, and the domain of $\varphi_{m}$ to be the $m^{\ths}$ \emph{partial recursive set} $W_{m} $  (also known as a \emph{recursively enumerable set}, abbreviated to r.e.~set). A presentation $P=\langle X|R\rangle$ is said to be a \emph{countably generated recursive presentation} if $X$ is a recursive enumeration of generators, and $R$ a recursive enumeration of relators. If $P,Q$ are group presentations then we denote their free product presentation by $P*Q$, given by taking the disjoint union of their generators and relators; this extends to the free product of arbitrary collections of presentations.
If $X$ is a set, we write $X^{*}$ for the set of finite words on $X \cup X^{-1}$, including the empty word $\emptyset$.  If $\phi: X \to Y^{*}$ is a set map, then we write $\overline{\phi}: X^{*} \to Y^{*}$ for the extension of $\phi$ to $X^{*}$.  If $g_{1}, \ldots, g_{n}$ are elements of a group $G$, then we write $\langle   g_{1}, \ldots, g_{n} \rangle^{G}$ for the subgroup in $G$ generated by these elements, and $\llangle g_{1}, \ldots, g_{n} \rrangle^{G}$ for the normal closure of these elements in $G$. \emph{Cantor's pairing function} is defined by $\left\langle \cdot , \cdot \right\rangle : \mathbb{N}\times\mathbb{N} \to \mathbb{N}$, $\left\langle x,y\right\rangle := \frac{1}{2}(x+y)(x+y+1)+y$, which gives a computable bijection.

\subsection{Embedding theorems}

\begin{defn}
Let $G$ be a group. We let $\ord(g)$ denote the order of a group element $g$, and say $g$ is \emph{torsion} if $1\leq \ord(g)<\infty$. We set 
\[
\tor(G):=\{g \in G\ |\ g \textnormal{ is torsion}\}
\] 
\[
\torord(G):=\{n \in \mathbb{N}\ |\ \exists g \in \tor(G) \textnormal{ with } \ord(g)=n\geq 2\}
\]
\end{defn}

\noindent Thus $\torord(G)$ is the set of orders of non-trivial torsion elements of $G$.

As detailed in \cite[Lemma 6.9 and Theorem 6.10]{Chiodo}, the following is implicit in Rotman's proof \cite[Theorem 12.18]{Rot} of the Higman embedding theorem.

\begin{thm}\label{tor emb}
There is a uniform algorithm that, on input of a countably generated recursive presentation $P=\langle X | R \rangle$, constructs a finite presentation $\T(P)$ such that $\overline{P}\hookrightarrow \overline{\T(P)}$ and  $\torord(\overline{P}) =\torord(\overline{\T(P)})$, along with an explicit embedding $\overline{\phi}: \overline{P} \hookrightarrow \overline{\T(P)}$.
\end{thm}

We will also use the following consequence to theorem \ref{tor emb}.

\begin{thm}[{\cite[Lemma 6.11]{Chiodo}}]\label{no re tor}
There is a uniform algorithm that, on input of any $n \in \mathbb{N}$, constructs a finite presentation $Q_{n}$ such that $\torord(\overline{Q}_{n})$ is one-one equivalent to $\mathbb{N}\setminus W_{n}$. Taking $n'$ with $W_{n'}$ non-recursive thus gives that $\torord(\overline{Q}_{n'})$ is not recursively enumerable; thus the set of finitely presented subgroups of $\overline{Q}_{n'}$ is not recursively enumerable.
\end{thm}

\section{Universal finitely presented torsion-free groups}\label{uni groups}

If $G,H$ are groups with $H$ torsion-free, a surjective homomorphism $h: G \twoheadrightarrow H$ is \emph{universal} if, for any  torsion-free $K$ and any homomorphism $f:G \to K$, there is a homomorphism $\phi:H \to K$ such that $f=\phi \circ h: G \to K$, i.e., the following diagram commutes:
\begin{displaymath}
\xymatrix{
G \ar[r]^{h} \ar[dr]_{f} & H \ar[d]^{\phi} \\
 & K
}
\end{displaymath}
Note that if $\phi$ exists then it will be unique. Indeed, if $\phi'$ also satisfies $f=\phi' \circ h$, then $\phi \circ h = \phi' \circ h$, and hence $\phi=\phi'$ as $h$ is a surjection and thus is right-cancellative. Moreover, any such $H$ is unique, up to isomorphism. Such an $H$ is called the \emph{universal torsion-free quotient} for $G$, denoted $G^{\tf}$. Observe that if $G$ is itself torsion-free, then $G^{\tf}$ exists and $G^{\tf} \cong G$, as the identity map $\id_{G}: G \to G$ has the universal property above.

A standard construction, showing that $G^{\tf}$ exists for every group $G$, is done via taking the quotient of $G$ by its \emph{torsion-free radical} $\rho(G)$, where $\rho(G)$ is the intersection of all normal subgroups $N\vartriangleleft G$ with $G/N$ torsion-free (see \cite{BrodHow}). It follows immediately that $G/\rho(G)$ has all the properties of a torsion-free universal quotient for $G$. 

We present here an alternative construction for $G^{\tf}$ which, though isomorphic to $G/\rho(G)$, lends itself more easily to an effective procedure for finitely (or recursively) presented groups, as shown in proposition \ref{tkill}.

\begin{defn}
Given a group $G$, we inductively define $\tor_{i}(G)$ as follows:
\[
\tor_{0}(G):=\{e\}
\]
\[
\tor_{i+1}(G):=\llangle\ \{g \in G\ |\ g\tor_{i}(G) \in \tor \big( G/\tor_{i}(G)\big) \}\ \rrangle ^{G}
\]
\[
\tor_{\infty}(G):=\bigcup_{i \in \mathbb{N}}\tor_{i}(G)
\]
Thus, $\tor_{i}(G)$ is the set of elements of $G$ which are annihilated upon taking $i$ successive quotients of $G$ by the normal closure of all torsion elements, and $\tor_{\infty}(G)$ is the union of all these. 
\end{defn}

By construction, $\tor_{i}(G) \leq  \tor_{j}(G)$ whenever $i \leq j$. It follows immediately that $\tor_{\infty}(G)\vartriangleleft G$. The finite presentation $P:=\langle x,y,z | x^2, y^3, xy=z^6 \rangle$ defines a group for which $\tor_{1}(\overline{P}) \neq \tor_{\infty}(\overline{P})$, as shown in \cite[Proposition 4.1]{Chi Vya}.

\begin{lem}\label{infinity tf}
If $G$ is a group, then $G/\tor_{\infty}(G)$ is torsion-free. 
\end{lem}

\begin{proof}
Suppose $g\tor_{\infty}(G) \in \tor \big(G/\tor_{\infty}(G)\big)$. Then $g^{n}\tor_{\infty}(G) =e$ in $G/\tor_{\infty}(G)$ for some $n>1$, so $g^{n} \in \tor_{\infty}(G)$. Thus there is some $i \in \mathbb{N}$ such that $g^{n} \in \tor_{i}(G)$, and hence $g\tor_{i}(G) \in \tor\big(G/\tor_{i}(G)\big)$. Thus $g \in \tor_{i+1}(G) \subseteq \tor_{\infty}(G)$, and so $g\tor_{\infty}(G) =e$ in $G/\tor_{\infty}(G)$.
\end{proof}

\begin{prop}
If $G$ is a group, then $\rho(G)=\tor_{\infty}(G)$.
\end{prop}

\begin{proof}
Clearly $\rho(G) \subseteq \tor_{\infty}(G) $, by definition of $\rho(G)$ and the fact that $G/\tor_{\infty}(G)$ is torsion-free (lemma \ref{infinity tf}). It remains to show that $\tor_{\infty}(G) \subseteq \rho(G)$. We proceed by contradiction, so assume $\tor_{\infty}(G) \nsubseteq \rho(G)$. Then there is some $N\vartriangleleft G$ with $G/N$ torsion-free, along with some minimal $i$ such that $\tor_{i}(G) \nsubseteq N$ (clearly, $i>0$, as $\tor_{0}(G) =\{e\}$). Then, by definition of $\tor_{i}(G)$ and the fact that $N$ is normal, there exists $e \neq g \in \tor_{i}(G)$ such that $g\tor_{i-1}(G) \in \tor \big( G/\tor_{i-1}(G)\big)$ and $g \notin N$ (or else $\tor_{i}(G) \subseteq N$). But then $g^{n} \in \tor_{i-1}(G)$ for some $n>1$. Since $\tor_{i-1}(G) \subseteq N$ by minimality of $i$, we have that $gN$ is a (non-trivial) torsion element of $G/N$, contradicting torsion-freeness of $G/N$. Hence $\tor_{\infty}(G) \subseteq \rho(G)$.
\end{proof}

\begin{cor}\label{tf quot}
If $G$ is a group, then $G/\tor_{\infty}(G)=G^{\tf}$; the torsion-free universal quotient for $G$.
\end{cor}

What follows is a standard result, which we state without proof.

\begin{lem}\label{words}
Let $P=\langle X | R \rangle$ be a countably generated recursive presentation. Then the set of words $\{w \in X^{*}|\ w =e \textnormal{ in } \overline{P}\}$ is r.e. 
\end{lem}

\begin{lem}\label{tor re}
Let $P=\langle X | R \rangle$ be a countably generated recursive presentation. Then the set of words $\{w \in X^{*}|\ w \in \tor(\overline{P}) \textnormal{ in } \overline{P}\}$ is r.e.
\end{lem}

\begin{proof}
Take any recursive enumeration $\{w_{1}, w_{2}, \ldots \}$ of $X^{*}$. Using lemma \ref{words}, start checking if $w_{i}^{n}=e$ in $\overline{P}$ for each $w_{i} \in X^{*}$ and each $n \in \mathbb{N}$ (by proceeding along finite diagonals). For each $w_{i}$ we come across which has a finite order, add it to our enumeration. This procedure will enumerate all words in $\tor(\overline{P})$, and only words in $\tor(\overline{P})$. Thus the set of words in $X^{*}$ representing elements in $\tor(\overline{P})$ is r.e.
\end{proof}

From this, we deduce the following:

\begin{lem}\label{tor i re}
Given a countably generated recursive presentation $P=\langle X | R \rangle$, the set $T_{i}:=\{w \in X^{*}|\ w\in \tor_{i}(\overline{P}) \textnormal{ in } \overline{P}\}$ is r.e., uniformly over all $i$ and all such presentations $P$. Moreover, the union $T_{\infty}:=\bigcup T_{i}$ is r.e., and is precisely the set $\{w \in X^{*}|\ w \in \tor_{\infty}(\overline{P}) \textnormal{ in } \overline{P}\}$.
\end{lem}

\begin{proof}
We proceed by induction. Clearly $\tor_{1}(\overline{P})$ is r.e., as it is the normal closure of $\tor(\overline{P})$, which is r.e.~by lemma \ref{tor re}. So assume that $\tor_{i}(\overline{P})$ is r.e.~for all $i \leq n$. Then $\tor_{n+1}(\overline{P})$ is the normal closure of $\tor(\overline{P}/\tor_{n}(\overline{P}))$, which again is r.e.~by the induction hypothesis and lemma \ref{tor re}. The rest of the lemma then follows immediately.
\end{proof}

\begin{prop}\label{tkill}
There is a uniform algorithm that, on input of a countably generated recursive presentation $P=\langle X | R \rangle$ of a group $\overline{P}$, outputs a countably generated recursive presentation $P^{\tf}=\langle X | R' \rangle$ (on the same generating set $X$, and with $R\subseteq R'$ as sets) such that $\overline{P^{\tf}}$ is the torsion-free universal quotient of $\overline{P}$, with associated surjection given by extending $\id_{X}: X \to X$.
\end{prop}

\begin{proof}
By corollary \ref{tf quot}, $\overline{P}^{\tf}$ is the group $\overline{P}/\tor_{\infty}(\overline{P})$. Then, with the notation lemma \ref{tor i re}, it can be seen that $P^{\tf}:=\langle X | R \cup T_{\infty} \rangle$ is a countably generated recursive presentation for $\overline{P}^{\tf}$, uniformly constructed from $P$.
\end{proof}

\begin{thm}\label{uni cbl tor free}
There is a finitely  presentable group $G$ which is torsion free, and contains an embedded copy of every countably generated recursively presentable torsion-free group.
\end{thm}

\begin{proof}
Take an enumeration $P_{1}, P_{2}, \ldots$ of all countably generated recursive presentations of groups, and construct the countably generated recursive presentation $Q:=P_{1}^{\tf}* P_{2}^{\tf}* \ldots$; this is the countably infinite free product of the universal torsion-free quotient of all countably generated recursively presentable groups (with some repetition). 
As each $P_{i}^{\tf}$ is uniformly constructible from $P_{i}$ (by proposition \ref{tkill}), we have that our construction of $Q$ is indeed effective, and hence $Q$ is a countably generated recursive presentation. 
Also, proposition \ref{tkill} shows that $\overline{Q}$ is a torsion-free group, as we have successfully annihilated all the torsion in the free product factors, and the free product of torsion-free groups is again torsion-free. Moreover, $\overline{Q}$ contains an embedded copy of every torsion-free countably generated recursively presentable group, as the universal torsion-free quotient of a torsion-free group is itself. Now use theorem \ref{tor emb} to embed $\overline{Q}$ into a finitely presentable group $\overline{\T(Q)}$. By construction, $\emptyset=\torord(\overline{Q})=\torord(\overline{\T(Q)})$, so $\overline{\T(Q)}$ is torsion-free. Finally, $\overline{\T(Q)}$ has an embedded copy of every countably generated recursively presentable torsion-free group, since $\overline{Q}$ did. Taking $G$ to be $\overline{\T(Q)}$ completes the proof.
\end{proof}

From this we immediately observe the following consequence.

\begin{thm}\label{uni tor free}
There is a universal finitely presented torsion-free group $G$. That is, $G$ is torsion-free, and for any finitely presented group $H$ we have that $H \hookrightarrow G$ if (and only if) $H$ is torsion-free.
\end{thm}

\begin{note}
One may ask why theorem \ref{uni tor free} does not follow immediately from Higman's embedding theorem by taking the free product of all finite presentations of torsion-free groups, and using the fact that Higman's theorem preserves orders of torsion elements. This cannot work, as we later shown in theorem \ref{pi2} that the set of finite presentations of torsion-free groups is not recursively enumerable.
\end{note}

\begin{rem}
Miller \cite[Corollary 3.14]{Mille-92}, extending a result of Boone and Rogers \cite[Theorem 2]{BooneRogers}, showed there is no universal finitely presented \emph{solvable word problem} group. It can be shown that none of the following group properties admit a universal finitely presented group: \emph{finite}, \emph{abelian}, \emph{solvable}, \emph{nilpotent} (\emph{simple}, however, remains open).
\end{rem}

\section{Complexity of embeddings}

Using the machinery described in section \ref{dunno}, we can encode the following recursion theory facts into groups.

\begin{lem}[{\cite[\S 13.2 Theorem VIII]{Rogers}}]\label{rec complete}
The set $\{ n \in \mathbb{N}\ | \ W_{n}=\mathbb{N}\}$ is $\Pi^{0}_{2}$-complete; 
the set $\{ n \in \mathbb{N}\ | \ |W_{n}|< \infty\}$ is $\Sigma^{0}_{2}$-complete.
\end{lem}

We can thus recover the following result, first proved in \cite[Main Theorem]{Lempp}.

\begin{thm}\label{pi2}
The set of finite presentations of torsion-free groups is $\Pi^{0}_{2}$-complete.
\end{thm}

\begin{proof}
Given $n \in \mathbb{N}$, we use theorem \ref{no re tor} to construct a finite presentation $Q_{n}$ such that $\torord(\overline{Q}_{n})$ is one-one equivalent to $\mathbb{N}\setminus W_{n}$. Thus $\overline{Q}_{n}$ is torsion-free if and only if $W_{n}=\mathbb{N}$. From lemma \ref{rec complete}, $\{ n \in \mathbb{N}\ | \ W_{n}=\mathbb{N}\}$ is $\Pi^{0}_{2}$-complete, so the set of torsion-free finite presentations is at least $\Pi^{0}_{2}$-hard. But this set has the following $\Pi^{0}_{2}$ description (taken from \cite{Lempp}):
\[
G \textnormal{ is torsion-free if and only if } \forall w\in G \forall n>0 (w^{n}\neq_{G} e \textnormal{ or } w=_{G} e)
\]
and hence is $\Pi^{0}_{2}$-complete.
\end{proof}

Combining this with the universal torsion-free group from theorem \ref{uni tor free}, we get the following immediate corollary, which extends theorem \ref{no re tor}.

\begin{cor}\label{pi2comp}
There is a finitely presented group whose finitely presentable subgroups form a $\Pi_{2}^{0}$-complete set.
\end{cor}

A similar construction to the proof of theorem \ref{no re tor} (as found in \cite[Lemma 6.11]{Chiodo}) gives us the following:

\begin{prop}\label{Cm embeds}
For any fixed prime $p$, the set of finite presentations into which $C_{p}$ embeds is $\Sigma^{0}_{2}$-complete.
\end{prop}

\begin{proof}
Given $n\in \mathbb{N}$ we form the countably generated recursive presentation $P_{n}$ as follows: 
\[
P_{n}:=  \langle x_{0}, x_{1}, \ldots\ | \ \{x_{i}^{p} \ | \ i \in \mathbb{N}\} \cup \{x_{0}, \ldots, x_{j}\ | \  j \in  W_{n}\}   \rangle
\]
If $|W_{n}|< \infty$ then $\overline{P}_{n} \cong C_{p}*C_{p}*\ldots$. Conversely, if $|W_{n}|= \infty$ then $\overline{P}_{n} \cong \{e\}$. So 

\[ \torord(\overline{P}_{n})= \left \{
 \begin{array}{l} 
    \{p\} \textnormal{ if } |W_{n}|< \infty \\
    \ \emptyset \ \textnormal{ if  } |W_{n}|= \infty \\
 \end{array}
\right.
\]
That is, $C_{p} \hookrightarrow \overline{P}_{n}$ if and only if $|W_{n}|< \infty$. Now use theorem \ref{tor emb} to construct a finite presentation $\T(P_{n})$ such that $\overline{P}_{n}\hookrightarrow \overline{\T(P_{n})} $ with $\torord(\overline{P}_{n})=\torord(\overline{\T(P_{n})})$. Hence $C_{p} \hookrightarrow \overline{\T(P_{n})}$ if and only if $|W_{n}|< \infty$, so by lemma \ref{rec complete} the set of finite presentations into which $C_{p}$ embeds is $\Sigma^{0}_{2}$-hard. But this set has the following straightforward $\Sigma^{0}_{2}$ description:
\[
C_{p} \hookrightarrow G \textnormal{ if and only if } \exists w\in G (w\neq_{G} e \textnormal{ and } w^{p}=_{G} e)
\]
and hence is $\Sigma^{0}_{2}$-complete.
\end{proof}

We can now prove:

\begin{thm}\label{main1}
Take an enumeration $P_{1}, P_{2}, \ldots$ of all finite presentations of groups; $P_{i} = \langle X_{i} | R_{i} \rangle$. Then the set $K=\{ (i , j) \in \mathbb{N}^{2}\ |\ \overline{P}_{i} \hookrightarrow \overline{P}_{j}\}$ is $\Sigma^{0}_{2}$-hard, $\Pi^{0}_{2}$-hard, and has a $\Sigma^{0}_{3}$ description.
\end{thm}

\begin{proof}
Corollary \ref{pi2comp} shows that $K$ is $\Pi^{0}_{2}$-hard, Proposition \ref{Cm embeds} shows that $K$ is $\Sigma^{0}_{2}$-hard, and the following is a $\Sigma^{0}_{3}$ description for $K$:
\[
K=\{ ( i , j ) \in \mathbb{N}^{2}\ |\ (\exists \phi : X_{i} \to X_{j}^{*})(\forall w \in X_{i}^{*}) (\overline{\phi}(w)=_{\overline{P}_{j}}e \textnormal{ if and only if } w=_{\overline{P}_{i}}e)\}
\]
\end{proof}

Note that, with the aid of Cantor's pairing function (a computable bijection between $\mathbb{N}^{2}$ and $\mathbb{N}$), we can view the set $K$ above as being a subset of $\mathbb{N}$. Hence it makes sense to talk of $K$ being $\Pi^{0}_{2}$-hard etc.

Based on theorem \ref{main1}, we conjecture the following:

\begin{conj}
The set $K$ defined above is $\Sigma^{0}_{3}$-complete. That is, the problem of deciding for finite presentations $P_{i}, P_{j}$ if $\overline{P}_{i} \hookrightarrow \overline{P}_{j}$ is $\Sigma^{0}_{3}$-complete.
\end{conj}

\section{Complexity of $\torord(G)$}\label{comp torord}

We now apply our techniques to investigate the complexity of $\torord(G)$ for $G$ a finitely presented group.

\begin{defn}
Call a set $A \subseteq \mathbb{N}_{\geq 2}$ \emph{factor-complete} if it is closed under taking non-trivial factors. That is, $n\in A   \Rightarrow  m \in  A$ for all $m>1$ with $m|n$.
\end{defn}

We give a set-theoretic description of precisely which factor-complete sets can appear as $\torord(G)$ for $G$ finitely (or recursively) presented. We presented an earlier proof of the following result in \cite{MO Chiodo}; what follows is a clearer proof pointed out to us by the anonymous referee.

\begin{thm}\label{torord equivalence}
For a set of natural numbers $A$ the following are equivalent:
\\$(1)$ $A=\torord(G)$ for some finitely presented group $G$;
\\$(2)$ $A=\torord(G)$ for some countably generated recursively presented group $G$;
\\$(3)$ $A$ is a factor-complete $\Sigma_{2}^{0}$ set.
\end{thm}

\begin{proof}
$(2)\Rightarrow (1)$ because, by theorem \ref{tor emb}, any recursively presented group can be embedded into a finitely presented group with the same $\torord$.

$(1)\Rightarrow (3)$. First, observe that $\torord(G)$ is factor-complete (for any group $G$), because if $\ord(g)=mn$ then $\ord(g^{m})=n$, for any $g \in G$. Second, $\torord(G)$ is a  $\Sigma_{2}^{0}$ set. Indeed, if $G$ has finite presentation $\langle X\ |\ R\rangle$, and $S$ is the set of words in $X^{*}$ which represent the trivial element in $G$, then
\[
\torord(G) = \{n \ | \ \exists w \in X^{*} ( n>1 \wedge w^{n} \in S \wedge \forall i (0<i<n \Rightarrow w^{i} \notin S) )  \}
\]
Since $S$ is r.e. (by lemma \ref{words}), it is a $\Sigma_{1}^{0}$ subset of $X^{*}$, and so the result follows.

$(3)\Rightarrow (2)$. As $A$ is a $\Sigma_{2}^{0}$ set, it has a description of the form
\[
A= \{n \in \mathbb{N}\ | \ \exists x \forall y R(n,x,y)\}
\]
for some ternary recursive relation $R$ on $\mathbb{N}$. Let
\[
 P := \{ (n,m) \in \mathbb{N}^{2}\ | \ (\forall x \leq m) (\exists y)\neg R(n,x,y) \}
\]
Clearly $P$ is r.e. If $n \notin A$ then $(n,m) \in P$ for all $m$. Conversely, if $n \in A$ then
\[
 (n,m) \in P \ \ \Leftrightarrow \ \ m<m_{n}:=\min \{ m\ | \ (\forall y)R(n,m,y) \}
\]
Let $I:=\{ (n,m) \in \mathbb{N}^{2} \ | \ n>1 \}$, and let $G:= \overline{\langle X|T \rangle}$ where
\[
X:= \{a_{nm} \ | \ (n,m) \in I\}, \ \ \ T:= \{a_{nm}^{n} \ | \ (n,m) \in I\} \cup \{a_{nm} \ | \ (n,m) \in I \cap P\}
\]
Clearly, $T$ is r.e., and so $G$ has countably generated recursive presentation. By the observations above, $G$ can be defined by the generators $a_{nm}$ and relators $a_{nm}^{n}=e$, where $n \in A$ and $m \geq m_{n}$. Let $K_{n}$ denote the free product of countably many cyclic groups $C_{n}$ of order $n$. Then $G$ is isomorphic to the free product
\[
G \cong *_{n \in A}K_{n}
\]
and therefore
\[
 \torord(G)= \bigcup_{n \in A}\torord(C_{n})=\bigcup_{n \in A}\{k \ | \ k|n \wedge k>1\}=A;
\]
the latter equality holds because $A$ is factor-complete.
\end{proof}

\begin{note} Theorem \ref{torord equivalence} was first proved in the more restricted setting of primes (i.e., considering sets of integers consisting only of primes) by Steinberg \cite{MO Steinberg} and Wilton \cite{MO Wilton}, in response to a question asked by Kohl \cite{MO Kohl}. Moreover, in the comments in \cite{MO Wilton}, Dorais gave a sketch of an alternate proof of the version for primes. Our original proof was a formalisation the proof by Dorais, and our result is an extention of this to the more general setting of all factor-complete $\Sigma_{2}^{0}$ sets. We thank Dorais, Kohl, Steinberg, and Wilton for their online discussion, as well as their insight into key aspects of this result; our work in this section is an extension of their ideas and results.
\end{note}

From the uniformity of the constructions in the proof of theorem \ref{torord equivalence}, we make the following observation.

\begin{prop}
The equivalence discussed in theorem $\ref{torord equivalence}$  is computable, in the following sense:
\\a$)$ Given a countably generated recursive presentation $Q$, we can compute from it a finite presentation $P$ with $\torord(\overline{P})=\torord(\overline{Q})$.
\\b$)$ Given a finite presentation $P$, we can compute from it a ternary recursive relation $R$ on $\mathbb{N}$ for which $\torord(\overline{P})=\{n \in \mathbb{N}\ | \ \exists x \forall y R(n,x,y)\}$.
\\c$)$ Given a ternary recursive relation $R$ on $\mathbb{N}$ for which $A:=\{n \in \mathbb{N}\ | \ \exists x \forall y R(n,x,y)\}$ is factor-complete, we can compute from it a countably generated recursive presentation $Q$ with $\torord(\overline{Q})=A$.
\end{prop}

We adopt the standard numbering of primes $\{p_{i}\}_{i \in \mathbb{N}}$, ordered by size; the following lemma is then immediate.

\begin{lem}\label{prime trick}
Let $X\subseteq \mathbb{N}$. Then the set
\[
X_{\prim}:=\{p_{i}\ |\ i \in X\} 
\]
is factor-complete and one-one equivalent to $X$.
\end{lem}

Applying lemma \ref{prime trick} to theorem \ref{torord equivalence}, we can conclude the following: 

\begin{cor}\label{any}
Given any $\Sigma_{2}^{0}$ set $A$, the set $A_{prime}$ is one-one equivalent to $A$, and can be realised as the set of orders of torsion elements of some finitely presented group $G$.
\end{cor}

\section{Further work}

This paper invites research into several questions. We mention some here.

\begin{fur}
Given the existence of a universal torsion-free group (theorem \ref{uni tor free}), and the constructions of Valiev \cite{val, val2} of explicit finite presentations of universal finitely presented groups, one could perhaps combine these techniques to produce an explicit finite presentation of a universal torsion-free group.
\end{fur}

\begin{fur}
The positions of the following properties in the arithmetic hierarchy have not been fully determined. Techniques such as those we have covered here may be of use in locating them.
\\1. \emph{Solvable}: Known to have a $\Sigma^{0}_{3}$ description.
\\2. \emph{Residually finite}: Known to have a $\Pi^{0}_{2}$ description.
\\3. \emph{Simple}: Known to have a $\Pi^{0}_{2}$ description.
\\4. \emph{Orderable}: Known to have a $\Pi^{0}_{3}$ description (the Ohnishi condition).
\\Points 1--3 are mentioned in \cite[p.~20]{Mille-92}, while point 4 appears in \cite[Lemma 2.2.1]{Glass}. We note that it may very well be the case that some of these are neither $\Pi^{0}_{n}$-complete nor $\Sigma^{0}_{n}$-complete, for any $n$.
\end{fur}

\begin{fur}
Following from theorem \ref{any}, and the uniformity of such a realisation of a $\Sigma_{2}^{0}$ set $A$ as one-one equivalent to the torsion orders of a finitely presented group, one could perhaps construct an explicit finite presentation $P$ of a group with $\torord(\overline{P})$ being $\Sigma_{2}^{0}$-complete, by encoding the set $\{ n \in \mathbb{N}\ | \ |W_{n}|< \infty\}$ which is $\Sigma^{0}_{2}$-complete (lemma \ref{rec complete}).
\end{fur}

\vspace{5pt}

\noindent \scriptsize{\textsc{Mathematics Department, University of Neuch\^{a}tel
\\Rue Emile-Argand 11, Neuch\^{a}tel, 2000, SWITZERLAND
\\maurice.chiodo@unine.ch}

\end{document}